\documentclass{amsart} 
\newtheorem{Thm}{Theorem}[section]
\newtheorem{Prop}[Thm]{Proposition}
\newtheorem{Cor}[Thm]{Corollary}
\newtheorem{Lem}[Thm]{Lemma}

\numberwithin{equation}{section}
\begin{document} 

\title[$p$-harmonic boundary and $D_p$-massive subsets]{The $p$-harmonic boundary and $D_p$-massive subsets of a graph of bounded degree}

\author[M. J. Puls]{Michael J. Puls}
\address{Department of Mathematics \\
John Jay College-CUNY \\
524 West 59th Street \\
New York, NY 10019 \\
USA}
\email{mpuls@jjay.cuny.edu}
\thanks{The research for this paper was partially supported by PSC-CUNY grant 63873-00 41}

\begin{abstract}
Let $p$ be a real number greater than one and let $\Gamma$ be a graph of bounded degree. We investigate links between the $p$-harmonic boundary of $\Gamma$ and the $D_p$-massive subsets of $\Gamma$. In particular, if there are $n$ pairwise disjoint $D_p$-massive subsets of $\Gamma$, then the $p$-harmonic boundary of $\Gamma$ consists of at least $n$ elements. We also show that the converse of this statement is also true.
\end{abstract}


\date{March 31, 2013}
\maketitle

\section{Introduction}\label{Introduction}
Throughout this paper $p$ will always denote a real number greater than one. A graph is said to have the {\em $p$-Liouville property} if every bounded $p$-harmonic function on the graph is constant. Similarly, a graph is said to have the {\em $D_p$-Liouville property} if every bounded $p$-harmonic function of the graph with finite $p$-Dirichlet sum is constant. When a graph has the $p$-Liouville property ($D_p$-Liouville property), the set of bounded $p$-harmonic functions (with finite $p$-Dirichlet sum) can be identified with $\mathbb{R}$, the real numbers. Now let $G$ be a finitely generated group. Our main motivation for studying the $p$-harmonic boundary of a graph arose from the problem of determining the first reduced $\ell^p$-cohomology space of $G$. A locally finite graph with bounded degree, called the Cayley graph of $G$, can be associated with $G$. Thus it makes sense to define the $p$-harmonic boundary for $G$, and to say that $G$ has the $p$-Liouville property ($D_p$-Liouville property). It turns out that the first reduced $\ell^p$-cohomology space of $G$ vanishes if and only if $G$ has the $D_p$-Liouville property if and only if the $p$-harmonic boundary of $G$ consists of one point or is empty. A more complete discussion about this characterization can be found in \cite{Puls10} and the references therein. Another reason for studying locally finite graphs with bounded degree is there intimate connection via discrete approximation to complete Riemannian manifolds with bounded geometry. The papers \cite{HoloSoar, kanaicapp, Kanai} contain a wealth of information concerning this link between graphs and manifolds.

Recently, a generalized version of the $D_p$-Liouville property for graphs has been studied in \cite{KimLee05, KimLee07}. More precisely, under what conditions on a graph can the bounded $p$-harmonic functions with finite $p$-Dirichlet sum be identified with $\mathbb{R}^n, n \in \mathbb{N}$. When $n \geq 2$, this also means that there are nonconstant $p$-harmonic functions on the graph. Holopainen and Soardi proved in \cite[Lemma 5.7]{HoloSoar} that there is a nonconstant bounded $p$-harmonic function with finite $p$-Dirichlet sum on a graph of bounded degree if and only if there exists two disjoint $D_p$-massive subsets of vertices of the graph. 

The purpose of this paper is to bring into sharper focus this connection between $D_p$-massive subsets and nonconstant $p$-harmonic functions on a graph. As a consequence, we are able to determine exactly when the set of bounded $p$-harmonic functions on a graph with finite $p$-Dirichlet sum can be identified with $\mathbb{R}^n$. The main tool we use to obtain our results is the $p$-harmonic boundary of a graph. 

The $p$-harmonic boundary is a subset of the $p$-Royden boundary. When $p = 2$ these sets are respectively known as the harmonic boundary and the Royden boundary. In \cite[Chapter 6]{Soardibook} the Royden and harmonic boundaries were studied for locally finite graphs of bounded degree. Many of the results in \cite[Chapter 6]{Soardibook} were translated from corresponding results on complete Riemannian surfaces. See \cite[Chapter 3]{SarioNakai70} for information about the Royden and harmonic boundaries in the setting of complete Riemannian surfaces. However, there are some major differences between these two cases. In \cite[Example 6.27]{Soardibook} it was shown that the Royden boundary and the harmonic boundary coincide for a locally finite graph of bounded degree that satisfies a strong isoperimetric inequality. This is in stark contrast with the complete Riemannian surface case. More precisely, if the harmonic boundary is removed from the Royden boundary of a complete Riemannian surface, then the resulting set is dense in the Royden boundary! See \cite[page 157]{SarioNakai70} for the details of this fact. Furthermore, if the graph is a $k$-regular tree, $k \geq 3$, then there are no isolated points in the harmonic boundary of the tree, \cite[ page 145]{Soardibook}.

The problem of explicitly computing the $p$-harmonic boundary of a locally finite graph of bounded degree appears to be quite difficult. The only result we can find in this direction is in the paper \cite{Wysoczanski} where it is shown that the Royden boundary of a 2-regular tree, which can be considered as a Cayley graph for the integers, is a quotient space of $\beta\mathbb{N}$, the Stone-\v{C}ech compactification of $\mathbb{N}$. In \cite[Chapter 7]{Puls10} the author gave some examples of finitely generated groups whose $p$-harmonic boundary is empty or contains exactly one point by using the fact that the first reduced $\ell^p$-cohomology of those particular groups is zero.

In Section \ref{preliminaries} we define the main concepts used in this paper. We also state our main result. Section \ref{proofmainresult} is devoted to the proof of the main result. We explain in Section \ref{elaboration} how our result extends the main result of \cite{KimLee07}.

I would like to thank the referee for several excellent suggestions that improved the exposition of this paper.

\section{Definitions and statement of main result}\label{preliminaries}
Let $\Gamma$ be a graph with vertex set $V_{\Gamma}$ and edge set $E_{\Gamma}$. We will write $V$ for $V_{\Gamma}$ and $E$ for $E_{\Gamma}$. For $x \in V, N_x$ will be the set of neighbors of $x$ and $deg(x)$ will denote the number of neighbors of $x$. We shall say that $\Gamma$ is of {\em bounded degree} if there exists a positive integer $k$ for which $deg(x) \leq k$ for every $x \in V$. A path $\gamma$ in $\Gamma$ is a sequence of vertices $x_1,x_2,\dots,x_n,\dots$ where $x_{i+1} \in N_{x_i}$ for $1 \leq i \leq n-1$ and $x_i \neq x_j$ if $i \neq j$. Note that all paths considered in this paper have no self-intersections. A graph is {\em connected} if any two distinct vertices of the graph are joined by a path. All graphs considered in this paper will be connected, of bounded degree with no self loops and have countably infinite number of vertices. By assigning length one to each edge of $\Gamma$, $V$ becomes a metric space with respect to the shortest path metric. We will denote this metric by $d(x,y)$, where $x,y \in V$. Thus $d(x,y)$ gives the length of the shortest path joining the vertices $x$ and $y$. For $S \subseteq V$, the outer boundary $\partial S$ of $S$ is the set of vertices in $V\setminus S$ with at least one neighbor in $S$, and $\vert S \vert$ will denote the cardinality of $S$. We use $1_V$ to represent the function that takes the value 1 on all elements of $V$. Finally, if $x \in V$ and $n \in \mathbb{N}$, the natural numbers, then $B_n(x)$ will denote the metric ball that contains all elements of $V$ that have distance less than $n$ from $x$.

We now proceed to define some function spaces that will be used in this paper. Let $S \subseteq V$ and let $f$ be a real-valued function on $S \cup \partial S$. We define the $p$-th power of the {\em gradient}, the {\em $p$-Dirichlet sum}, and the {\em $p$-Laplacian} of $x \in S$ by
\begin{equation*}
\begin{split}
 \vert Df (x) \vert^p  & = \sum_{y \in N_x} \vert f(y) - f(x) \vert^p,   \\
 I_p (f, S)            & = \sum_{x \in S} \vert Df (x) \vert^p, \\
  \Delta_p f (x)    & = \sum_{y \in N_x} \vert f(y) - f(x) \vert^{p-2} (f(y) - f(x)).
\end{split}
\end{equation*} 
In the case $1 < p < 2$, we make the convention that $\vert f(y) - f(x) \vert^{p-2} (f(y) - f(x)) = 0$ if $f(y) = f(x)$. A function $f$ is said to be {\em $p$-harmonic} on $S$ if $\Delta_p f(x) = 0$ for all $x \in S$. Observe that a function which is $p$-harmonic on $S$ is also defined on $\partial S$. We now give an alternate definition that is commonly used for a function to be $p$-harmonic on $S$ when $S$ is a finite (compact) subset of $V$. We begin by setting
\[ \Xi (f,S) = \frac{1}{2} \left( I_p(f,S) + \sum_{x \in \partial S} \sum_{y \in N_x \cap S} \vert f(x) - f(y) \vert^p\right). \]
A function $f$ is said to be $p$-harmonic on $S$ if it is the minimizer of $\Xi$ among the functions in $S \cup \partial S$ with the same value in $\partial S$ as $f$, that is , if 
\[ \Xi (f,S) \leq \Xi (u,S) \]
for every function $u$ in $S \cup \partial S$ with $f = u$ in $\partial S$. The interested reader can find more information about $p$-harmonic functions and harmonic functions on graphs in the papers \cite{Canton01, HoloSoar, HolopainenSoardi97, Kaufman03, KimLee05, KimLee07, Puls10, Shanmugalingam03, Soardibook, Yamasaki1977} and the references therein.

We shall say that $f$ is {\em $p$-Dirichlet finite} if $I_p(f,V) < \infty$. The set of all $p$-Dirichlet finite functions on $\Gamma$ will be denoted by $D_p(\Gamma)$. With respect to the following norm $D_p(\Gamma)$ is a reflexive Banach space,
$$ \Vert f \Vert_{D_p} = \left( I_p(f,V) + \vert f(o) \vert^p \right)^{1/p},$$
where $o$ is a fixed vertex of $\Gamma$ and $f \in D_p(\Gamma)$. We use $HD_p(\Gamma)$ to represent the set of $p$-harmonic functions on $V$ that are contained in $D_p(\Gamma)$. Note that the constant functions are members of $HD_p(\Gamma)$. Let $\ell^{\infty}(\Gamma)$ denote the set of bounded functions on $V$ and let $\| f \|_{\infty} = \sup_V \vert f \vert$ for $f \in \ell^{\infty}(\Gamma)$. Set $BD_p(\Gamma) = D_p(\Gamma) \cap \ell^{\infty}(\Gamma)$. The set $BD_p(\Gamma)$ is a Banach space under the norm 
$$\Vert f \Vert_{BD_p} = \left( I_p(f, V)\right)^{1/p} + \Vert f \Vert_{\infty},$$
where $f \in BD_p(\Gamma)$. Let $BHD_p(\Gamma)$ be the set of bounded $p$-harmonic functions contained in $D_p(\Gamma)$. The space $BD_p(\Gamma)$ is closed under the usual operations of scalar multiplication, addition and pointwise multiplication. Furthermore, for $f, g \in BD_p(\Gamma)$ we have that $\Vert fg \Vert_{BD_p} \leq \Vert f \Vert_{BD_p} \Vert g \Vert_{BD_p}$. Thus $BD_p(\Gamma)$ is a commutative Banach algebra. Let $C_c(\Gamma)$ be the set of functions on $V$ with finite support. Indicate the closure of $C_c(\Gamma)$ in $D_p(\Gamma)$ by $\overline{C_c(\Gamma)}_{D_p}$. Set $B(\overline{C_c(\Gamma)}_{D_p}) = \overline{C_c(\Gamma)}_{D_p} \cap \ell^{\infty}(\Gamma)$. Using the fact that the inequality $(a+b)^{1/p} \leq a^{1/p} + b^{1/p}$ is true when $a,b \geq 0$ and $1 < p \in \mathbb{R}$, we see immediately that $\Vert f \Vert_{D_p} \leq \Vert f \Vert_{BD_p}$. Consequently, $B(\overline{C_c(\Gamma)}_{D_p})$ is closed in $BD_p(\Gamma)$. 

\subsection{The $p$-harmonic boundary}\label{defnpharmbound}
In this subsection we construct the $p$-harmonic boundary of a graph $\Gamma$. For a more detailed discussion about this construction see Section 2.1 of \cite{Puls10}. Let $Sp(BD_p(\Gamma))$ denote the set of complex-valued characters on $BD_p(\Gamma)$, that is the nonzero ring homomorphisms from $BD_p(\Gamma)$ to $\mathbb{C}$. We will implicitly use the following property of elements in $Sp(BD_p(\Gamma))$ throughout the paper.
\begin{Lem} \label{realvalued} Let $\chi \in Sp(BD_p(\Gamma))$. If $f \in BD_p(\Gamma)$, then $\chi(f)$ is a real number.
\end{Lem}
\begin{proof}
Suppose there exists an $f \in BD_p(\Gamma)$ for which $\chi (f) = a + bi$, where $b \neq 0$. Set $F = (f-a)/b$ and observe that $\chi(F) = i$. Since $BD_p(\Gamma)$ is a Banach algebra, $F, F^2$ and $F^2 + 1_V$ all belong to $BD_p(\Gamma)$. Also, $\chi(F^2 + 1_V) = 0$. For $x \in V$ and $y \in N_x$,
\[ \big\vert \frac{1}{F^2(y) + 1_V} - \frac{1}{F^2(x) + 1_V} \big\vert^p \leq \vert F^2(x) - F^2(y)\vert^p, \]
because $F^2 + 1_V \geq 1$ on $V$. It now follows that $(F^2 + 1_V)^{-1} \in BD_p(\Gamma)$, and so $F^2 + 1_V$ has a multiplicative inverse in $BD_p(\Gamma)$. Hence, $\chi(F^2 + 1_V) \neq 0$, a contradiction. Therefore, $\chi(f)$ is a real number.
\end{proof}

With respect to the weak $\ast$-topology, $Sp(BD_p(\Gamma))$ is a compact Hausdorff space. If $A \subseteq Sp(BD_p(\Gamma)), \overline{A}$ will indicate the closure of $A$ in $Sp(BD_p(\Gamma))$. Given a topological space $X$, let $C(X)$ denote the ring of continuous functions on $X$ endowed with the sup-norm. The Gelfand transform defined by $\hat{f}(\chi) = \chi(f)$ yields a monomorphism of Banach algebras from $BD_p(\Gamma)$ into $C(Sp(BD_p(\Gamma)))$ with dense image. Furthermore, the map $i \colon V \rightarrow Sp(BD_p(\Gamma))$ given by $(i(x))(f) = f(x)$ is an injection, and $i(V)$ is an open dense subset of $Sp(BD_p(\Gamma))$. For the rest of this paper we shall write $f$ for $\hat{f}$, where $f \in BD_p(\Gamma)$. The {\em $p$-Royden boundary} of $\Gamma$, which we shall denote by $R_p(\Gamma)$, is the compact set $Sp(BD_p(\Gamma))\setminus i(V)$. The {\em $p$-harmonic boundary} of $\Gamma$ is the following subset of $R_p(\Gamma)$:
\[  \partial_p(\Gamma) \colon = \{ \chi \in R_p(\Gamma) \mid \hat{f}(\chi) = 0 \mbox{ for all }f \in B(\overline{C_c(\Gamma)}_{D_p}) \}. \]
We shall write $\vert \partial_p (\Gamma) \vert$ to indicate the cardinality of $\partial_p(\Gamma)$.

\subsection{$D_p$-massive sets}\label{Dpmassive}
We now define the concept of a $D_p$-massive subset of a graph. An infinite connected subset $U$ of $V$ with $\partial U \neq \emptyset$ is called a $D_p$-massive subset of $V$ if there exists a nonnegative function $u \in BD_p(\Gamma)$ with the following properties:
\begin{enumerate} 
    \item $\Delta_p u(x) =0$ for $x\in U$,
  \item $u(x) = 0$ for $x \in \partial U$, and 
  \item $\sup_{x \in U} u(x) = 1$.
    \end{enumerate}
We call any $u$ that satisfies these conditions an {\em inner potential} of the $D_p$-massive subset $U$. The next result is Proposition 4.11 of \cite{Puls10} and will be needed in the sequel.
\begin{Prop} \label{Dpmassiveclosure}
If $U$ is a $D_p$-massive subset of $V$, then $\overline{i(U)}$ contains at least one point of $\partial_p(\Gamma)$.
\end{Prop}

\subsection{Statement of the main result}\label{statementmainresult}
 The main result of this paper is:
\begin{Thm}\label{mainresult}
Let $1 < p \in \mathbb{R}$ and let $\Gamma$ be a graph of bounded degree. Suppose $n \in \mathbb{N}$. Then there exists $n$ pairwise disjoint $D_p$-massive subsets $D_1, D_2, \dots ,D_n$ of $V$ if and only if $\vert \partial_p(\Gamma) \vert \geq n$.
\end{Thm}

By combining this theorem with Corollary 2.7 of \cite{Puls10} we obtain
\begin{Cor} \label{cormainresult}
Let $1 < p \in \mathbb{R}, n \in \mathbb{N}$ and let $\Gamma$ be a graph of bounded degree. If there exists $n$ pairwise disjoint $D_p$-massive subsets of $V$, but there does not exist $n+1$ disjoint $D_p$-massive subsets of $V$, then $BHD_p(\Gamma)$ can be identified with $\mathbb{R}^n$.
\end{Cor}

\section{Proof of Theorem \ref{mainresult}}{\label{proofmainresult}
The following lemma will be needed for the proof of Theorem \ref{mainresult}. For a proof of the lemma see the first part of the proof of \cite[Lemma 5.7]{HoloSoar}
\begin{Lem}\label{componentdpmassive}
Let $h$ be a nonconstant function in $BHD_p(\Gamma)$, and let $U$ be an infinite connected subset of $V$. Let $a$ and $b$ be real numbers such that
\[ \inf_{x \in U} h < a < b < \sup_{x \in U} h. \]
Then each component of the set $\{ x \in U \mid h(x) > b \}$ and each component of $\{ x \in U \mid h(x) < a \}$ is $D_p$-massive.
\end{Lem}

\subsection{Proof of Theorem \ref{mainresult}}

We are now ready to prove Theorem \ref{mainresult}. Let $D_1, D_2, \dots, D_n$ be a collection of pairwise disjoint $D_p$-massive subsets of $V$. For each $k$, with $1 \leq k \leq n$, let $u_k$ be an inner potential for $D_k$. We may and do assume $u_k = 0$ on $V \setminus D_k$. Also, $\overline{D_k} \cap \partial_p(\Gamma) \neq \emptyset$ by Proposition \ref{Dpmassiveclosure}. For each $k$ we will produce an element $\chi_k \in \overline{D_k} \cap \partial_p(\Gamma)$ for which $\chi_k(u_k) \neq 0$ and $\chi_k(u_j) = 0$ if $j \neq k$. This will establish $\vert \partial_p (\Gamma) \vert \geq n$. Extend $u_k$ to a continuous function on $Sp(BD_p(\Gamma))$. By \cite[Theorem 2.6]{Puls10} there exists a $p$-harmonic function $h_k$ on $V$ such that $h_k = u_k$ on $\partial_p(\Gamma)$. The maximum principle (\cite[Theorem 4.7]{Puls10}) says that $0 < h_k < 1$ on $V$. Let $B_k = \{ x \in D_k \mid h_k(x) > 1 - \epsilon\}$, where $0 < \epsilon < \frac{1}{4}$. Since $\sup u_k = 1$ on $D_k, B_k \neq \emptyset$. Let $C_k$ be a component of $B_k$. By Lemma \ref{componentdpmassive} $C_k$ is $D_p$-massive. Thus $\overline{C_k} \cap \partial_p(\Gamma) \neq \emptyset$. Select $\chi_k \in \overline{C_k} \cap \partial_p(\Gamma)$. Because $h_j = u_j$ on $\partial_p(\Gamma), \chi_k(h_j) = \chi_k(u_j)$. Consequently, $\chi_k(u_k) =1$ and $\chi_k(u_j) =0$ if $k \neq j$. Hence, $\vert \partial_p(\Gamma) \vert \geq n$ if there exists $n$ pairwise disjoint $D_p$-massive subsets of $V$.

Conversely, let $\chi_1, \chi_2, \dots, \chi_n$ be distinct elements from $\partial_p(\Gamma)$. By Urysohn's lemma there exists a continuous function $f_1 \colon Sp(BD_p(\Gamma )) \rightarrow [0,1]$ with $f_1(\chi_1) = 1$ and $f_1(\chi_k) = 0$ if $k \neq 1$. Let $M_1 = f_1^{-1}(1)$. For each integer $k$ with $2 \leq k \leq n$ we can inductively define a continuous function $f_k \colon Sp(BD_p(\Gamma)) \rightarrow [0,1]$ with the following properties: \\
\[ f_k(x) = \left\{ \begin{array}{cl} 
                          1   &   x = \chi_k \\
                         0    &   x= \chi_i, i \neq k \\
                         0    &   x \in \cup_{i=1}^{k-1} M_i  \end{array} \right. \] 
where $M_k = f_k^{-1} (1)$. 

By the density of $BD_p(\Gamma)$ in $C(Sp (BD_p (\Gamma)))$, we can assume $f_k \in BD_p(\Gamma)$ for each $k$. Using Theorems 4.6 and 4.8 of \cite{Puls10}, we obtain a unique $h_k \in BHD_p (\Gamma)$ with $h_k = f_k$ on $\partial_p(\Gamma)$ for each $k$. Also, $0 < h_k < 1$ on $V$. Observe that if $h_k(\chi) = 1 = h_j(\chi)$ for some $\chi \in \partial_p(\Gamma)$, then $k=j$. Let $\epsilon >0$ and consider the set $A_{k,\epsilon} = \{ x \in V \mid h_k(x) > 1 - \epsilon \}$. For each $k$ let $D_{k, \epsilon}$ be a component of $A_{k, \epsilon}$. Furthermore, choose the $D_{k, \epsilon}$ so that $D_{k,\epsilon_1} \subseteq D_{k, \epsilon_2}$ if $0 < \epsilon_1 < \epsilon_2$. Lemma \ref{componentdpmassive} yields that $D_{k, \epsilon}$ is $D_p$-massive. The proof will be complete if there exists an $\epsilon > 0$ such that $D_{k, \epsilon} \cap D_{j, \epsilon} = \emptyset$ if $k \neq j$. Assume for the purposes of contradiction that this condition is not true. Then there exists $j, k$ with $D_{k, \epsilon} \cap D_{j, \epsilon} \neq \emptyset$ for all $\epsilon > 0$. Let $ i \in \mathbb{N}$. Denote by $C_i$ a component of $D_{k, 2^{-i}} \cap D_{j, 2^{-i}}$. By the comparison principle \cite[Theorem 3.14]{HoloSoar} $C_i$ is infinite. Using Lemma \ref{componentdpmassive} we can produce a $D_p$-massive subset of $C_i$. An appeal to Proposition \ref{Dpmassiveclosure} produces a $\psi_i \in \overline{C_i} \cap \partial_p(\Gamma)$. Clearly $\psi_i(h_j) > 1 - 2^{-i}$ and $\psi_i(h_k) > 1 - 2^{-i}$. The sequence $(\psi_i)$ in $\partial_p(\Gamma)$ has a convergent subsequence that converges to some $\psi$ in $\partial_p(\Gamma)$. Consequently, $\psi(h_k) = 1= \psi(h_j)$. This contradicts our earlier observation that if $h_k(\chi) = 1 = h_j(\chi)$ for some $\chi \in \partial_p(\Gamma)$, then $k = j$. Therefore, there exists an $\epsilon >0$ for which $D_{k, \epsilon} \cap D_{j, \epsilon} = \emptyset$ for each $j, k$ with $1 \leq j, k \leq n$. The proof of the theorem is now complete.

\section{A result of Kim and Lee}\label{elaboration} 

In this section we elaborate on how Theorem \ref{mainresult} improves the main result of \cite{KimLee07}. We start by giving some needed definitions.

Recall that $E$ represents the edge set of a graph $\Gamma$. Denote by $\mathcal{F}(E)$ the set of all real-valued functions on $E$ and let $\mathcal{F}^+(E)$ be the subset of $\mathcal{F}(E)$ that consists of all nonnegative functions. For $f \in \mathcal{F}(E)$ set 
\[ \xi_p(f) = \sum_{e \in E} \vert f(e) \vert^p. \]
The edge set of a path $\gamma$ in $\Gamma$ will be denoted by $Ed(\gamma).$ Let $Q$ be a set of paths with no self-intersections in $\Gamma$. Indicate by $\mathcal{A}(Q)$ the set of all $f \in \mathcal{F}^+(E)$ that satisfy $\xi_p(f) < \infty$ and $\sum_{e \in Ed(\gamma)} f(e) \geq 1$ for all $\gamma \in Q$. The {\em extremal length} of order $p$ for $Q$ is defined by 
\[ \lambda_p(Q)^{-1} = \inf\{ \xi_p(f) \mid f \in \mathcal{A}(Q) \}. \]
The number $\lambda_p(Q)^{-1}$ is commonly known as the {\em $p$-modulus} of the path family $Q$. We shall say that a property holds for {\em $p$-almost every path} in a collection of paths if the set of paths for which the property does not hold has infinite extremal length (or $p$-modulus zero).

Let $A \subseteq V$, write $\Gamma_A$ for the largest subgraph of $\Gamma$ that has vertex set $A$. Let $\gamma$ be a one-sided infinite path in $\Gamma$. For a real-valued function $f$ on $V$, set $f(\gamma) = \lim_{n \rightarrow \infty} f(x_n)$ as $n \rightarrow \infty$ along the vertices of $\gamma$. Let $P_A$ be the set of all one-sided infinite paths with no self-intersections contained in $\Gamma_A$. We define a real-valued function $f$ to be {\em asymptotically constant} on $A$ if there exists a constant $c$ such that 
\[  f(\gamma) = c \mbox{ for } p\mbox{-almost every path } \gamma \in P_A. \]
We shall say that an infinite connected set $U$ has {\em property AC} if each function in $BHD_p(\Gamma)$ is asymptotically constant on $U$.

An infinite connected subset $S$ of $V$ is said to be {\em $p$-hyperbolic} if there exists a nonempty finite subset $A$ of $V$ for which 
\[   Cap_p(A, \infty, S) = \inf_u I_p(u,S) > 0, \]
where the infimum is taken over all finitely supported functions $u$ on $S \cup \partial S$ such that $u =1$ on $A$. If $S$ is not $p$-hyperbolic, then it is said to be {\em $p$-parabolic}. The quantity $Cap_p(A, \infty, S)$ is known as the {\em $p$-capacity} of $S$.

Motivated by \cite[Theorem 3.1]{Yamasaki86}, Kim and Lee prove the following result in \cite[Theorem 1.1]{KimLee07}
\begin{Thm} \label{KimLeeresult1} 
Let $n \in \mathbb{N}$ and let $\Gamma$ be a graph with $n$ $p$-hyperbolic ends. Suppose each $p$-hyperbolic end has property $AC$. Then given any real numbers $a_1, a_2, \dots, a_n \in \mathbb{R}$, there exists an unique $h \in BHD_p(\Gamma)$ such that
\[ h(\gamma) = a_i \mbox{ for $p$-almost every path } \gamma \in P_{F_i} \]
for each $i = 1, 2, \dots , n$, where $F_1, F_2, \dots, F_n$ are the $p$-hyperbolic ends of $\Gamma$.
\end{Thm}

We see immediately that if a graph $\Gamma$ satisfies the hypothesis of this theorem, then $BHD_p(\Gamma)$ can be identified with $\mathbb{R}^n$, which is the same conclusion as Corollary \ref{cormainresult}. However, the hypothesises of Theorem \ref{KimLeeresult1} are quite strong. The number of ends of a graph $\Gamma$ is independent of $p$, and the $AC$ property is also very restrictive. For example, let $G$ denote a co-compact lattice in  the real rank 1 simple Lie groups $Sp(n,1), n \geq 2$. The Cayley graph of the group $G$ has one end, but there are nonconstant $p$-harmonic functions with finite $p$-Dirichlet sum on $G$ exactly when $ p > 4n +2$. See \cite[Section 4]{Puls06} for the details.

When the cardinality of $\partial_p(\Gamma)$ is finite, Theorem \ref{mainresult} completely characterizes the number of elements in $\partial_p(\Gamma)$ in terms of pairwise disjoint $D_p$-massive sets. It is the case that $D_p$-massive sets are also $p$-hyperbolic. The reason we are able to drop the property $AC$ assumption from Theorem \ref{KimLeeresult1}  in our Theorem \ref{mainresult} is given in Proposition \ref{onepoint} below. Before we prove the proposition we need the following
\begin{Lem}\label{onepointprep}
Let $\Gamma$ be a graph with bounded degree and let $1 < p \in \mathbb{R}$. Suppose $F$ is an infinite connected subset of $V$ with property $AC$. For $h \in BHD_p(\Gamma)$, denote by $c_h$ the constant for which $h(\gamma) = c_h$ for $p$-almost every path in $P_F$. If $\chi \in \overline{F} \cap \partial_p(\Gamma)$, then $\chi(h) = c_h$.
\end{Lem}

\begin{proof} 
Let $h \in BHD_p(\Gamma)$. Suppose $c_h < \chi(h)$. Let $\epsilon > 0$ such that $c_h < \chi(h) - \epsilon$. Define $A = \{ x \in F \mid h(x) > \chi(h) - \epsilon \}$ and let $C$ be a component of $A$. Observe that $\lambda_p (P_C) = \infty$ due to $h(\gamma) > c_h$ for each $\gamma \in P_C$. By Lemma \ref{componentdpmassive}, $C$ is $D_p$-massive. Proposition 5.3 of \cite{PulsSciFenn} now yields the contradiction $\lambda_p(P_C) < \infty$. A similar argument shows that it is also not the case $\chi(h) < c_h$. Therefore $\chi(h) = c_h$.
\end{proof}

Denote by $V(\gamma)$ the vertex set of an infinite path $\gamma$ in $\Gamma$. Write $\overline{V}(\gamma)$ for the closure of $i(V(\gamma))$ in $Sp(BD_p(\Gamma))$. The set of extreme points of $\gamma$ is given by
 \[ Ex(\gamma) = \overline{V}(\gamma) \setminus i(V(\gamma)). \]
\begin{Prop}\label{onepoint}
Let $1 < p \in \mathbb{R}$ and let $\Gamma$ be a graph of bounded degree. Let $F$ be a $p$-hyperbolic subset of $V$. Then $F$ has property $AC$ if and only if $\vert \overline{F} \cap \partial_p(\Gamma) \vert = 1$.
\end{Prop}
\begin{proof}
Because $F$ is $p$-hyperbolic, it is the case $\lambda_p (P_F) < \infty$. Lemma 5.2 of \cite{PulsSciFenn} implies $\overline{F} \cap \partial_p(\Gamma) \neq \emptyset$. Now suppose $\chi_1$ and $\chi_2$ are distinct elements from $\overline{F} \cap \partial_p(\Gamma)$. Since $BD_p(\Gamma)$ separates points in $Sp(BD_p(\Gamma))$, there exists an $f \in BD_p(\Gamma)$ for which $\chi_1(f) \neq \chi_2(f)$. Combining Theorems 4.6 and 4.8 of \cite{Puls10} we obtain an $h \in BHD_p(\Gamma)$ with the property $f = h$ on $\partial_p(\Gamma)$. Thus $\chi_1(h) \neq \chi_2(h)$, contradicting Lemma \ref{onepointprep}. Hence, $\vert \overline{F} \cap \partial_p(\Gamma) \vert = 1$.

Now assume $\vert \overline{F} \cap \partial_p(\Gamma) \vert = 1$ and let $\chi$ be the unique element in $\overline{F} \cap \partial_p(\Gamma)$. Select an $h \in BHD_p(\Gamma)$ and let $c_h = \chi(h)$. We will now show that $h(\gamma) = c_h$ for $p$-almost every path in $P_F$. Denote by $P_{\infty}$ the set of all $\gamma \in P_F$ for which $h(\gamma)$ does not exist. Let $\gamma = x_0x_1 \dots x_n \ldots \in P_{\infty}$. The identity $h(x_n) = h(x_0) - \sum_{k=1}^n (h(x_{k-1}) - h(x_k))$ implies $\sum_{k=1}^{\infty} \vert h(x_{k-1}) - h(x_k) \vert = \infty$. It now follows \cite[Lemma 2.3]{KayanoYamasaki84} that $\lambda_p(P_{\infty}) = \infty$. For each $n \in \mathbb{N}$, set
\[ P_{1/n} = \{ \gamma \in P_F \setminus P_{\infty} \mid \vert h(\gamma) - c_h \vert > 1/n\}. \]
Now suppose $\lambda_p(P_{1/n}) < \infty$ for some $n \in \mathbb{N}$. By \cite[Lemma 5.2]{PulsSciFenn} 
\[ ( \overline{\cup_{\gamma} \{ Ex(\gamma) \mid \gamma \in P_{1/n} \} }) \cap \partial_p(\Gamma) \neq \emptyset. \]
Let $\psi$ be an element in this intersection. The definition of $P_{1/n}$ implies that $\psi(h) \neq c_h$. Combining the fact $P_{1/n} \subseteq P_F$ with the hypothesis $\vert \overline{F} \cap \partial_p(\Gamma) \vert = 1$ yields $\psi = \chi$, contradicting the fact $\chi(h) = c_h$. Hence $\lambda_p(P_{1/n}) = \infty$ for all $n \in \mathbb{N}$. Let $P_U = \cup_{n=1}^{\infty} P_{1/n}$. Lemma 2.2 of \cite{KayanoYamasaki84} says that $\lambda_p(P_U) = \infty$, and $\lambda_p(P_U \cup P_{\infty}) = \infty$. Let $P_h = \{ \gamma \in P_F \mid h(\gamma) = c_h \}$. Then  $P_F = P_h \cup P_U \cup P_{\infty}.$ Another appeal to \cite[Lemma 2.2]{KayanoYamasaki84} shows that $\lambda_p(P_h) < \infty$ since $\lambda_p(P_F) < \infty$. Thus $h(\gamma) = c_h$ for $p$-almost every path in $P_F$. Therefore, $h$ is asymptotically constant on $F$.
\end{proof}

It follows immediately from this proposition that if a graph $\Gamma$ satisfies the assumptions of Theorem \ref{KimLeeresult1}, then $\vert \partial_p(\Gamma) \vert = n$.

\bibliographystyle{plain}
\bibliography{newpharmbddpmassive}
\end{document}